\documentclass[10pt, leqno, a4paper]{amsart}
\usepackage{tikz}
\usepackage{float}
\usepackage{xcolor}
\title[A nonsmooth path-connectedness relation in the real plane]{A nonsmooth path-connectedness relation in the real plane}
\author{YUSUF UYAR}
\address{Department of Mathematics, Middle East Technical University, 06800, Ankara, Turkey.}
\email{yuyar@metu.edu.tr}
\usepackage{hyperref}
\usepackage{amsmath, amssymb, amsthm}
\newtheorem{theorem}{Theorem}
\newtheorem*{question}{Question}
\newtheorem*{theoremx}{Main Theorem}

\newcommand{\C}{\mathcal{C}}
\newcommand{\e}{E^*_0}

\newcommand{\cantor}{2^{\mathbb{N}}}
\newcommand{\N}{\mathbb{N}}
\newcommand{\R}{\mathbb{R}}
\newcommand{\balpha}{\boldsymbol{\alpha}}
\newcommand{\bbeta}{\boldsymbol{\beta}}

\begin{document}

\setlength{\abovedisplayskip}{0pt}
\setlength{\belowdisplayskip}{0pt}
\setlength{\abovedisplayshortskip}{0pt}
\setlength{\belowdisplayshortskip}{0pt}

\sloppy

\begin{abstract}
In this paper, we show that path-connectedness equivalence relation of the Knaster continuum which is a compact subset of the real plane is Borel bireducible to a nonsmooth hyperfinite Borel equivalence relation. This answers a question of \cite{bec98}.
\end{abstract}
\maketitle
\section{Introduction}
Let $X$ be a Polish space, i.e. a completely metrizable separable topological space. Consider the path-connectedness equivalence relation $\approx _X$ on $X$ given by $x\approx _Xy$ iff there exists a continuous function $\gamma$ from the unit interval $\mathbb{I}$ into $ X$ with $\gamma(0)=x\text{ and }\gamma(1)=y$
for any $x,y\in X$. Then the equivalence relation $\approx _X$ is an analytic subset of $X \times X$ as it is the projection of the closed subset 
$$\{ (x,y,\gamma)\in X\times X\times \Pi (X):\gamma (0)=x\land \gamma (1)=y\}$$ where $\Pi (X)$ is the Polish space of all continuous functions from $\mathbb{I}$ to $X$.

This paper is a contribution to the study of Borel complexity of path-connectedness relations of various subsets of $\mathbb{R}^n$. While there have been some results on this theme, several questions seem to not have gotten enough attention. In particular, regarding the descriptive complexity of the equivalence relation of path-connectedness on compact subsets of $\R ^n$,  Becker \cite{bec98} asked the following two questions:
\begin{itemize}
    \item Is $\approx _K$ Borel for any compact subset $K$ of the real plane?
\item Is there a compact subset $K \subseteq \mathbb{R}^2$ such that $\approx _K$ is not smooth?
\end{itemize}

In regard to the first question, an example of a compact subset $K\subseteq \R^3$ for which $\approx _K$ is non-Borel had been given in \cite{kun82} before Becker's question. It was then shown in \cite{bec01} that if $X\subseteq \R^2$ is a $G_{\delta}$ subset, then $\approx _X$ being Borel is equivalent to each path-component of $X$ being Borel. Recently, a more powerful result that $\approx _X$ is Borel for any $G_{\delta}$ subset $X \subseteq \mathbb{R}^2$ is proven in \cite{debs24}, answering the first question positively. 

In this paper, we show that the Knaster continuum which was constructed by Bronisław Knaster in 1922 actually answers the second question of Becker affirmetively (see \cite{knaster}). More precisely, we prove the following.

\begin{theoremx}
\label{main}
   There exists a compact subset $H\subseteq \mathbb{R}^2$ for which $\approx _H$ is Borel bireducible to a nonsmooth hyperfinite Borel equivalence relation.
\end{theoremx}
In Section 2, we shall provide some necessary background on theory of Borel equivalence relations and relevant notation for the construction. In Section 3, we prove the main theorem and conclude with a next naturally arising question.\\

\textbf{Acknowledgements.} This paper is a part of the author's master's thesis \cite{uyar} which was partially supported by TÜBİTAK BİDEB (2210/E Program) and written under the supervision of Burak Kaya, to whom the author would like to extend his sincere thanks for his helpful advice and guidance on author's research work. The author would like to thank Benjamin Vejnar for letting him know that the construction done here is a well-known example that had been done long before. 

\section{Preliminaries}
\subsection{Background on Borel equivalence relations}\label{subsec} Let $X$ and $Y$ be Polish spaces and let $E \subseteq X \times X$ and $F \subseteq Y \times Y$ be Borel equivalence relations. A Borel map $f:X\rightarrow Y$ is called a \textit{Borel reduction} from $E$ to $F$ if
$$xEy \longleftrightarrow f(x)Ff(y)$$
for any $x,y\in X$. In the case that there is a Borel reduction from $E$ to $F$, we say that $E$ is \textit{Borel reducible} to $F$, in which case we write  $E\leq _B F$. We say that $E$ and $F$ are Borel bireducible if $E\leq _B F$ and $F \leq _B E$, in which case we write $E \sim_B F$. The notion of Borel reducibility is a tool to compare relative complexities of equivalence relations on Polish spaces.

A Borel equivalence relation $E$ on a Polish space $X$ is said to be \textit{smooth} if we have $E \leq_B \Delta_Y$ for some Polish space $Y$, where $\Delta_Y$ is the identity relation on $Y$. An important example of a nonsmooth Borel equivalence relation is the eventual equality relation $E_0$ on the Cantor space $\cantor$ given by  $$\alpha E_0\beta\longleftrightarrow \exists m\: \forall n\geq m\:  (\alpha _n=\beta _n)$$ for any $\alpha , \beta \in \cantor$, for example, see \cite[Section 6.1]{gao}. Indeed, $E_0$ is the least complex nonsmooth Borel equivalence relation with respect to Borel reducibility in the sense that it is the $\leq_B$-successor of $\Delta_{\mathbb{R}}$ \cite{HKL90}.

A Borel equivalence relation $E$ is said to be \textit{hyperfinite} if $E=\bigcup_n R_n$ for some increasing sequence $R_0 \subseteq R_1 \subseteq \dots$ of Borel equivalence relations with finite equivalence classes. For example, $E_0$ is the increasing union of Borel equivalence relations $F_n$ with finite classes given by
\[ \alpha\ F_n\ \beta\ \longleftrightarrow\ \alpha_i=\beta_i \text{ for all } i \geq n\ \] and so it is hyperfinite. It was proven in \cite{djk} that any two nonsmooth hyperfinite Borel equivalence relations on Polish spaces are Borel bireducible. Consequently, up to Borel bireducibility, $E_0$ is the unique nonsmooth hyperfinite equivalence relation.\\ 

For each $\alpha \in \cantor$, let $\overline{\alpha}$ denote the sequence obtained by taking the complement of each entry, that is, $\overline{\alpha}_n=1-\alpha_n$. Consider the equivalence relation $\e$ on $\cantor $ given by $$\alpha \e \beta \longleftrightarrow(\alpha E_0 \beta\ \text{ or }\ \alpha E_0 \overline{\beta})$$ for any $\alpha ,\beta \in \cantor$. It is straightforward to check that $\e$ is a hyperfinite Borel equivalence relation. Moreover, the continuous map $f:\cantor \rightarrow \cantor$ defined as $f(\alpha)=(\alpha _0,1,\alpha_1,1,\alpha_2,1,...)$ for all $\alpha \in \cantor$ is a Borel reduction from $E_0$ into $\e$, which shows that $\e$ is nonsmooth. Consequently, $\e \sim_B E_0$.

\subsection{Notations} Given a Polish space $X$, we shall denote the set of all non-empty compact subsets of $X$ by $\mathcal{K}(X)$, which itself becomes a Polish space when it is endowed with the Vietoris topology.

In the rest of the paper, let $\C$ denote the Cantor set $$\C:=\left\{ \displaystyle \sum _{i=0}^{\infty } \frac{2\alpha _i}{3^i}\in [0,1]: \alpha _i\in \{0,1\} \text{ for all }i\in \N\right\}.$$ 
Recall that $\C$ is homeomorphic to the Cantor space $\cantor$ via the map
\[ \alpha \mapsto \displaystyle \sum _{i=0}^{\infty } \frac{2\alpha _i}{3^{i+1}}\ .\]
For any $\alpha \in \cantor$, we shall denote its image under this homeomorphism by the boldface letter $\mathbf{\balpha}$. We set $N_0=\cantor$ and for any $k\in \mathbb{N}^+$, we define $N_k\subseteq \cantor $ as $$N_k=\left\{ \alpha\in \cantor  :\min \{i\in \N:\alpha _i=1\}=k-1\right\} \ .$$
For any sequence $\alpha \in N_k$ where $k>0$,  let $\widehat{\alpha}$ be the sequence $(0,0,...,0,1,\overline{\alpha}_{k+1},\overline{\alpha}_{k+2},\overline{\alpha}_{k+3},...)$, which is still in $N_k$. Lastly, $\overline{\balpha}$ and $\widehat{\balpha}$ will be the real numbers corresponding to the sequences  $\overline{\alpha}$ and $\widehat{\alpha}$ respectively.

\section{Proof of the Main Theorem}

We shall first define a sequence $(H_k)_{k}$ of compact subsets of $\mathbb{R}^2$ as follows. Let $x_0=\frac{1}{2}$. For any $\alpha \in N_0 $, let $\gamma ^0_{\alpha,\overline{\alpha}}$ be the upper semicircle centered at $(x_0,0)$ with radius $|x_0-\balpha|$ connecting $(\balpha,0)$ to $(\overline{\balpha},0)$. Set
\[H_0=\displaystyle \bigcup _{\alpha \in N_0 }\gamma ^0_{\alpha ,\overline{\alpha}}\]
For each $k>0$, let $x_k=\frac{\balpha +\widehat{\balpha}}{2}$ where $\alpha$ is some (equivalently, any) element of $N_k$. Then for each $\alpha \in N_k$, let $\gamma ^k_{\alpha ,\widehat{\alpha}}$ be the lower semicircle centered at $(x_k,0)$ with radius $|x_k-\balpha|$ connecting $(\balpha  ,0)$ to $(\widehat{\balpha},0)$. Then set \[H_k=\displaystyle \bigcup _{\alpha \in N_k}\gamma ^k_{\alpha ,\widehat{{\alpha}}}\]
Set $H:=\bigcup_kH_k$. Indeed, the compact set $H$ is shown in Figure $\ref{figure}$.

For an example, let $\alpha =(0,1,1,0,...)$ and $\beta =(0,0,0,1,...)$ where $\alpha F_4\overline{\beta }$ and $F_4$ is defined as in the Subsection \ref{subsec} so that $\alpha \e \beta$. Then the diagram
$$\alpha =(0,1,1,0,...) \overset{\gamma ^0_{\alpha,\overline{\alpha}}}{\longleftrightarrow}(1,0,0,1,...) \overset{\gamma ^1_{\overline{\alpha}, \widehat{\overline{\alpha}}}}{\longleftrightarrow}(1,1,1,0,...) \overset{\gamma ^0_{\widehat{\overline{\alpha}},\overline{\widehat{\overline{\alpha}}}}}{\longleftrightarrow}(0,0,0,1,...)=\overline{\widehat{\overline{\alpha}}}=\beta $$
shows the path connecting $(\balpha ,0)$ to $(\bbeta ,0)$, which is $\gamma ^0_{\alpha,\overline{\alpha}}\cup \gamma ^1_{\overline{\alpha}, \widehat{\overline{\alpha}}}\cup \gamma ^0_{\widehat{\overline{\alpha}},\overline{\widehat{\overline{\alpha}}}}$.
 \vspace*{-3cm}
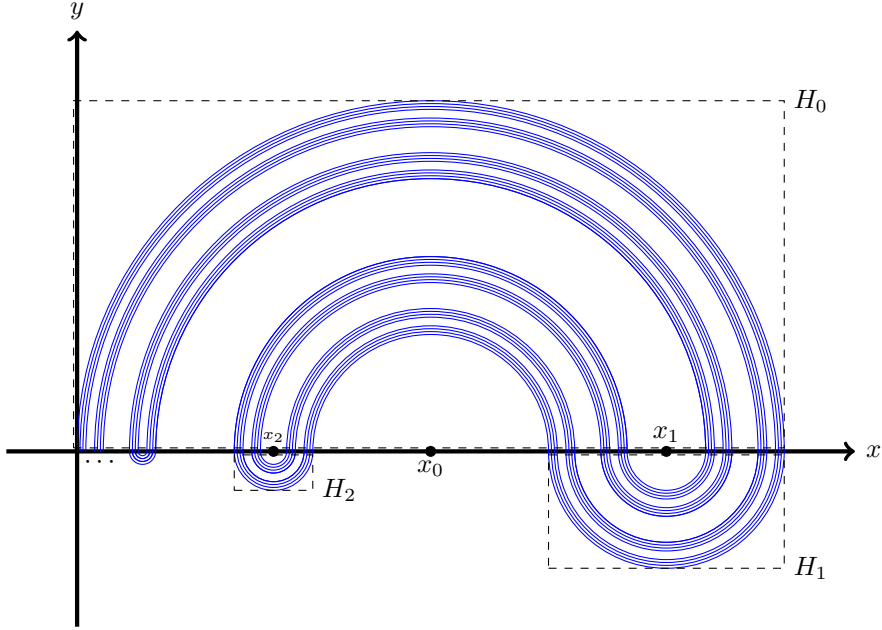
\begin{figure}[H]
\begin{tikzpicture}[scale=9.3]
\centering

\draw[->,ultra thick] (-0.1,0)--(1.1,0) node[right]{$x$};
\draw[->,ultra thick] (0,-0.25)--(0,0.6) node[above]{$y$};

\begin{scope}
\clip (0,0) rectangle (1,1);
 
\draw[blue] (0.5,0) circle(0.5-0/3);
\draw[blue] (0.5,0) circle(0.5-1/3);

\draw[blue] (0.5,0) circle(0.5-1/9);
\draw[blue] (0.5,0) circle(0.5-2/9);
\draw[blue] (0.5,0) circle(0.5-7/9);
\draw[blue] (0.5,0) circle(0.5-8/9);

\draw[blue] (0.5,0) circle(0.5-1/27);
\draw[blue] (0.5,0) circle(0.5-2/27);
\draw[blue] (0.5,0) circle(0.5-7/27);
\draw[blue] (0.5,0) circle(0.5-8/27);
   
\draw[blue] (0.5,0) circle(0.5-1/81);
\draw[blue] (0.5,0) circle(0.5-2/81);
\draw[blue] (0.5,0) circle(0.5-7/81);
\draw[blue] (0.5,0) circle(0.5-8/81);
\draw[blue] (0.5,0) circle(0.5-19/81);
\draw[blue] (0.5,0) circle(0.5-20/81);
\draw[blue] (0.5,0) circle(0.5-25/81);
\draw[blue] (0.5,0) circle(0.5-26/81);    

\draw[blue] (0.5,0) circle(0.5-1/243);
\draw[blue] (0.5,0) circle(0.5-2/243);
\draw[blue] (0.5,0) circle(0.5-7/243);
\draw[blue] (0.5,0) circle(0.5-8/243);
\draw[blue] (0.5,0) circle(0.5-19/243);
\draw[blue] (0.5,0) circle(0.5-20/243);
\draw[blue] (0.5,0) circle(0.5-25/243);
\draw[blue] (0.5,0) circle(0.5-26/243); 
\draw[blue] (0.5,0) circle(0.5-55/243); 
\draw[blue] (0.5,0) circle(0.5-56/243);
\draw[blue] (0.5,0) circle(0.5-61/243);
\draw[blue] (0.5,0) circle(0.5-62/243);
\draw[blue] (0.5,0) circle(0.5-73/243);
\draw[blue] (0.5,0) circle(0.5-74/243);
\draw[blue] (0.5,0) circle(0.5-79/243);
\draw[blue] (0.5,0) circle(0.5-80/243); 
\end{scope}

\begin{scope}[shift={(0.66666,0)},scale=1/3]
\clip (0,0) rectangle (1,-1);
 
\draw[blue] (0.5,0) circle(0.5-0/3);
\draw[blue] (0.5,0) circle(0.5-1/3);
\draw[blue] (0.5,0) circle(0.5-1/9);
\draw[blue] (0.5,0) circle(0.5-2/9);
\draw[blue] (0.5,0) circle(0.5-7/9);
\draw[blue] (0.5,0) circle(0.5-8/9);
\draw[blue] (0.5,0) circle(0.5-1/27);
\draw[blue] (0.5,0) circle(0.5-2/27);
\draw[blue] (0.5,0) circle(0.5-7/27);
\draw[blue] (0.5,0) circle(0.5-8/27);
\draw[blue] (0.5,0) circle(0.5-1/81);
\draw[blue] (0.5,0) circle(0.5-2/81);
\draw[blue] (0.5,0) circle(0.5-7/81);
\draw[blue] (0.5,0) circle(0.5-8/81);
\draw[blue] (0.5,0) circle(0.5-19/81);
\draw[blue] (0.5,0) circle(0.5-20/81);
\draw[blue] (0.5,0) circle(0.5-25/81);
\draw[blue] (0.5,0) circle(0.5-26/81); 
\end{scope}

\begin{scope}[shift={(0.222222,0)},scale=1/9]
\clip (0,0) rectangle (1,-1);
 
\draw[blue] (0.5,0) circle(0.5-0/3);
\draw[blue] (0.5,0) circle(0.5-1/3);
\draw[blue] (0.5,0) circle(0.5-1/9);
\draw[blue] (0.5,0) circle(0.5-2/9);
\draw[blue] (0.5,0) circle(0.5-7/9);
\draw[blue] (0.5,0) circle(0.5-8/9);
\draw[blue] (0.5,0) circle(0.5-1/27);
\draw[blue] (0.5,0) circle(0.5-2/27);
\draw[blue] (0.5,0) circle(0.5-7/27);
\draw[blue] (0.5,0) circle(0.5-8/27);
  \end{scope}

\begin{scope}[shift={(0.074,0)},scale=1/27]
\clip (0,0) rectangle (1,-1);
 
\draw[blue] (0.5,0) circle(0.5-0/3);
\draw[blue] (0.5,0) circle(0.5-1/3);
\draw[blue] (0.5,0) circle(0.5-1/9);
\draw[blue] (0.5,0) circle(0.5-2/9);
\draw[blue] (0.5,0) circle(0.5-7/9);
\draw[blue] (0.5,0) circle(0.5-8/9);
\end{scope}

\draw[dashed] (-0.005,0.005) rectangle (1,0.5) node[anchor=west]{$H_0$};
\draw[dashed] (0.66666,-0.005) rectangle (1,-0.3333*0.5) node[anchor=west]{$H_1$};
\draw[dashed] (2/9,-0.005) rectangle (0.3333,-0.3333*0.3333*0.5) node[anchor=west]{$H_2$};
\draw  node at (0.035,-0.015) {$\cdots$} ;
\filldraw (0.5,0) circle(0.2pt) node[anchor=north]{$x_0$};
\filldraw (5/6,0) circle(0.2pt) node[anchor=south]{$x_1$};
\filldraw (5/18,0) circle(0.2pt) node[anchor=south]{\tiny $x_2$};

\end{tikzpicture}
\caption{The Knaster continuum}
\label{figure}
\end{figure}
\newpage
Since $\e$ is the unique hyperfinite nonsmooth Borel equivalence relation up to Borel bireducibility, the Main Theorem is an immediate consequence of the following.\\

\begin{theorem}
    $\approx _H$ is Borel bireducible to $E^*_0$.
\end{theorem}

\begin{proof} To show that $E^*_0 \leq _B\;\approx _H$, let $\psi :\cantor \rightarrow H$ be defined as $\psi (\alpha )=(\balpha,0)$, which is clearly continuous. We claim that $\alpha \e \beta$ if and only if $\psi (\alpha )\approx _H\psi (\beta)$.\\

Clearly, $\alpha \e\beta$ if and only if $\alpha F_n\beta$ or $\alpha F_n\overline{\beta}$ for some $n\in \mathbb{N}$.  We claim then that for any $n \in \mathbb{N}$ and for any $\alpha ,\beta\in \cantor $ with $\alpha F_n\beta$ or $\alpha F_n\overline{\beta}$, there is a path in $H$ connecting $\psi (\alpha )$ and $\psi (\beta )$. We shall prove this by induction on $n$.

For the base step $n=0$, the claim trivially holds whenever $\alpha =\beta$ and, $\psi(\alpha )$ and $\psi(\beta )$ are connected by the path $\gamma ^0_{\alpha,\overline{\alpha}}$ whenever $\overline{\alpha} =\beta$.

Let $n \in \mathbb{N}$ and assume that the claim holds for $n$. Let $\alpha ,\beta \in \cantor$ be such that $\alpha F_{n+1}\beta$ or $\alpha F_{n+1}\overline{\beta}$. Suppose that we are in the case that $\alpha F_{n+1}\beta$. If $\alpha _n=\beta_n$, then we have $\alpha F_n\beta$ and so we are done by the inductive hypothesis. So, assume $\alpha _n\neq \beta _n$. Without loss of generality, assume that $\alpha _n=1$. We then have 

\begin{align*}
   \psi(\alpha )&=\psi (\alpha _0,...,\alpha _{n-1},1,\alpha _{n+1},\alpha _{n+2},...)\\ &\approx _H \psi (0   ,...,0,1,\alpha _{n+1},\alpha _{n+2},...) \\
   &\approx _H \psi (0   ,...,0,1,\overline{\alpha }_{n+1},\overline{\alpha }_{n+2},...) \\
&\approx _H \psi (1  ,..., 1 ,0,\alpha _{n+1},\alpha _{n+2},...) \\
   &\approx _H \psi (\beta _0,...,\beta _{n-1},0,\alpha _{n+1},\alpha _{n+2} , ...)=\psi(\beta )\\
\end{align*} 
where the first and the last equivalences follow from the inductive hypothesis whereas the second and the third ones are witnessed by some paths in $H_{n+1}$ and $H_0$ respectively. The case for $\alpha F_{n+1}\overline{\beta}$, though slightly more involved, can be shown in a similar way. Thus, we have $\alpha \e \beta$ implies $\psi (\alpha )\approx _H\psi (\beta)$.

For the other direction, assume that $\alpha,\beta\in \cantor $ such that $\alpha  \mathord{\not\mathrel{E}}_0^*\beta$. Assume for the sake of contradiction that $\psi (\alpha)\approx _H \psi (\beta)$. Say $\gamma:\mathbb{I}\rightarrow H$ is a path from $\psi (\alpha)$ to $\psi (\beta)$. We will first prove that $\gamma (\mathbb{I})$ must intersect infinitely many $H_k$'s. One can check that, by construction, for any $i\in \N$, the set $H_0\cup...\cup H_i$ is homeomorphic to $\C \times \mathbb{I}$. It is therefore easy to see that $\psi (\alpha)$ and $\psi (\beta)$ are not path-connected inside $H_0\cup ...\cup H_i$ for any $i\in \N$. Thus $\gamma (\mathbb{I})$ must intersect infinitely many $H_k$'s. Now choose $e_{n}\in H_{k_n}\cap \gamma (\mathbb{I})$ for each $n\in \N$ from distinct $H_{k_n}$'s. Take $t_n\in \mathbb{I}$ such that $\gamma (t_n)=e_{n}\in H_{k_n}$ for all $n\in \N$. Without loss of generality, we may assume that $(t_n)_n$ is an increasing (or decreasing) sequence. For if it is not the case, we can pass to an increasing (or decreasing) subsequence. For any $n\in \N$, $\gamma ([t_n,t_{n+1}])$ must intersect the line $x=\frac{1}{2}$, call $M$. This is because $\gamma (t_n)\in H_{k_n}$ and $\gamma (t_{n+1})\in H_{k_{n+1}}$ cannot be connected by a path inside $H\setminus M$. So for each $n\in \N$ let $\hat{t}_n\in (t_n,t_{n+1})$ such that $\gamma(\hat{t}_n)\in M$. Since $(t_n)_n$ and $(\hat{t}_n)_n$ converge to the same number and $\gamma$ is a continuous function, we have \[(0,0)=\displaystyle \lim _{n\rightarrow \infty} e_{n}=\lim _{n\rightarrow \infty}\gamma (t_n)=\lim _{n\rightarrow \infty}\gamma (\hat{t}_n)\in M,\]  which is clearly not true. Thus, $\psi (\alpha )\not \approx_H\psi (\beta)$. This finishes the proof that $\e \leq  _B\ \approx _H$.

To show $\approx _H\ \leq  _B\ \e$, consider the map $\varphi: H \rightarrow \cantor$ that sends each $e \in H$ to the sequence $\varphi(e)$ in $\cantor$ that maps under $\psi$ to the closest point on the copy of the Cantor set on the $x$-axis, to which $e$ is path-connected. Here we choose the left point in the case that there are two such points. It is straightforward to check that $\varphi$ is a Borel reduction from $\approx _H$ to $\e$.
\end{proof}
We constructed a compact subset of the real plane for which path-connectedness equivalence relation is Borel bireducible to $E_0$, which is the $\leq _B-$least nonsmooth Borel equivalence relation. We can ask then if this is the maximal possible complexity of $\approx _K$ for  $K\in \mathcal{K}( \mathbb{R}^2)$. That is to ask: does there exist a compact subset $K\subseteq \mathbb{R}^2$ for which $\approx _K$ is strictly more complicated than $E_0$?
\begin{question}
    Is there a compact subset $K$ of $\mathbb{R}^2$ such that $E_0 < _B\ \approx _K$? 
\end{question}

\vspace{1cm}
\bibliographystyle{amsalpha}
\bibliography{refs}

\end{document}